\documentclass[preprint,12pt]{article}
\usepackage{amssymb}
 \usepackage{amsthm}
\usepackage{latexsym}
\usepackage{amsmath}
\usepackage{amssymb}
\usepackage{color}

 \newcommand{\udots}{\mathinner{\mskip1mu\raise1pt\vbox{\kern7pt\hbox{.}}
\mskip2mu\raise4pt\hbox{.}\mskip2mu\raise7pt\hbox{.}\mskip1mu}}
\usepackage{amsmath}
\newtheorem{definition}{Definition}
\newtheorem{theorem}{Theorem}[section]

\newtheorem{lemma}{Lemma}[section]
\newtheorem{remark}{Remark}[section]

\begin{document}
\title{\bf \Large {\bf Quantifying Dependence Between Random Vectors: A New Index with Applications}}
{{\author{\normalsize{Chuancun Yin}\\{\normalsize\it    (School of Statistics and Data Science, Qufu Normal University}\\
\noindent{\normalsize\it Shandong 273165, China)}}}}
\maketitle
\vskip0.01cm
\noindent{\large {\bf Abstract}}  {This article proposes a new index for quantifying the degree of dependence between random vectors. The index takes values in [0,1] and equals zero if and only if the random vectors are sub-independent. Unlike mere uncorrelatedness, sub-independence implies a stronger form of dependence  while remaining strictly weaker than full independence. The proposed index is constructed via characteristic functions and admits a simplified representation in terms of moments. We establish its theoretical properties and derive a computationally efficient formula for the corresponding empirical measure. Furthermore, we investigate the asymptotic behavior of the estimator and demonstrate its practical utility through applications in machine learning, actuarial science, and renewal theory.}



\baselineskip =20pt

\numberwithin{equation}{section}
\section{Introduction}\label{intro}
In the era of big data,  measuring and testing   the dependence between  random vectors play a central role
in many fields, including statistics,   biology and engineering, data science, machine learning, among
others. Up so far some methods for measuring and testing dependence have
appeared in the literature.
The most popular correlation coefficient is the   Pearson
product-moment correlation which   measures the linear dependence between two random variables.
Except for the Pearson product-moment correlation, the other two  most popular classical measures of statistical association are   Spearman's $\rho$, and Kendall's $\tau$. These coefficients are very powerful for detecting linear or monotone associations, however, it does not fully characterize dependence. To overcome
those limitations, Szekely, Rizzo, and Bakirov (2007) proposed a   nonparametric distance covariance
(dCov) and distance correlation (dCor)   to measure dependence  between two random vectors of arbitrary
dimensions, not necessarily
the same dimension; see also Szekely  and Rizzo (2009). Moreover, they also extended it to more general a class of measures called    $\alpha$-distance dependence measures.
Further, Sz\'ekely and Rizzo (2013) proposed a $t$-test based on a modified distance covariance. The research progress in this topic can be found in recent papers  Shao and Zhang (2014), Yao et al. (2018), Zhu et al. (2020), Li et al. (2024), Borgonovo et al. (2025),  Liu  and Shang (2026), and the references therein.
For two random vectors  $X\in \Bbb{R}^p$ and $Y\in \Bbb{R}^q$,  according to  Szekely, Rizzo, and Bakirov (2007), the  $\alpha$-distance $(0<\alpha<2)$ covariance  between $X$ and $Y$ is defined as
\begin{eqnarray}
{\rm dCov}^{2(\alpha)}(X,Y)=\frac{1}{c(p,\alpha) c(q,\alpha)}\int_{\Bbb{R}^{p+q}}\frac{|\phi_{X,Y}(t,s)-\phi_X(t)\phi_Y(s)|^2}{|t|_p^{\alpha+p}|s|_q^{\alpha+q}}dtds,
\end{eqnarray}
where  $\phi_X(t)$ and $\phi_Y(s)$   are the characteristic functions of $X$ and $Y$, respectively,  $\phi_{X,Y}$ is the joint characteristic
function of $(X,Y)$, and
$$c(p,\alpha)=\frac{2\pi^{p/2}\Gamma(1-\alpha/2)}{\alpha 2^{\alpha}\Gamma((\alpha+p)/2)},\,\,  c(q,\alpha)=\frac{2\pi^{q/2}\Gamma(1-\alpha/2)}{\alpha 2^{\alpha}\Gamma((\alpha+q)/2)}. $$
Here $\Gamma(\cdot)$ is the complete gamma function. In the simplest case, $\alpha =1$, the constants in (1.1) become
 $$c_p=c(p,1)=\frac{\pi^{(1+p)/2}}{\Gamma((1+p)/2)}, \,\,  c_q=c(q,1)=\frac{\pi^{(1+q)/2}}{\Gamma((1+q)/2)},$$
 which are half the surface areas of the unit spheres in $\Bbb{R}^p$ and $\Bbb{R}^q$, respectively.
  The integral in (1.1) exists provided that $X$ and $Y$ have finite  absolute moments of  $\alpha$-order.
Write ${\rm dCov}^{2(\alpha)}(X)={\rm dCov}^{2(\alpha)}(X,X)$ and ${\rm dCov}^{2(\alpha)}(Y)={\rm dCov}^{2(\alpha)}(Y,Y)$. Analogously to  Pearson correlation,
distance correlation (dCor) is the normalized version of
distance covariance, which is defined as
\begin{eqnarray*}
	{\rm dCor}^{2(\alpha)}(X,Y) =\left\{\begin{array}{ll}
		\frac{{\rm dCov}^{2(\alpha)}(X,Y)}{\sqrt{{\rm dCov}^{2(\alpha)}(X){\rm dCov}^{2(\alpha)}(Y)}}, \ &{\rm if}\, \, {\rm dCov}^{2(\alpha)}(X){\rm dCov}^{2(\alpha)}(Y)>0,\\
		 0,  \ &{\rm if}\,\, {\rm dCov}^{2(\alpha)}(X){\rm dCov}^{2(\alpha)}(Y)=0.
	\end{array}
	\right.
\end{eqnarray*}
 Szekely et al. (2007) showed that  for distributions with finite first
moments, we have $0\le {\rm dCor}^{2(\alpha)}(X,Y)\le 1$, and  ${\rm dCor}^{2(\alpha)}(X,Y)= 0$ if and only if $X$ and $Y$ are mutually independent.
Hence, distance correlation provides a natural extension of  Pearson correlation and rank correlation in
capturing arbitrary types of dependence.     The extension of distance covariance from pairs of random
variables to $n$-tuplets of random variables  was addressed in   B\"ottcher et al. (2019). Interested readers are referred to Szekely et al. (2007),  Szekely  and Rizzo (2009) and  B\"ottcher et al. (2019) for more details of the back-ground of the distance covariance.


Uncorrelatedness is quite weak in comparison with independence.
So, we are faced with this question: is there
a concept which is stronger than uncorrelatedness but
weaker than independence?  The answer to this question is yes and it is called sub-independence. This concept was first introduced by  Durairajan (1979) and subsequently used   in many papers, see e.g. Hamedani and Volkmer (2009), Ebrahimi et al. (2010), Hamedani (2013),  Schennach(2019), Dong et al. (2022)  and references therein.
Sub-independence, although less intuitive than independence in appearance, provides a more flexible and realistic theoretical foundation in many practical modeling scenarios. Since it only requires that the distribution of the sum of variables can be calculated through convolution, rather than requiring that the complete joint distribution can be decomposed into a product form, it has specific applications in multiple fields.
However, there are no  indices  to measure this dependence,  just like the product moment covariance and correlation measure the linear dependence of random variables as well as the distance covariance and distance correlation measure all types of dependencies between random variables.  Inspired by the ground breaking work  on distance covariance in  Szekely et al. (2007),  this paper focuses on a new dependence coefficient that measures such the dependence   for two random vectors  $X$ and $Y$ in arbitrary dimensions. The main features of our coefficient are the following: it has a simple expression; it belongs to the range
of  [0,1] and the value zero only if the random vectors
are sub-independent; it is fully nonparametric;   although
our measure is based on the  characteristic function, it can be simplified as an expression of  moments.
We study its properties, both theoretical and empirical, of the new  coefficient. As  applications   we present several examples such as the concentration inequalities used in machine learning, renewal theorems and sums of randomly indexed sequences in the framework of sub-independence.

The rest of this paper is organized as follows.  In Section 2  we propose a new  measure and
   the theoretical properties     are discussed.   In Section 3 the comparisons of sub-independence   with known dependencies are investigated  and results for the bivariate normal and Cauchy are given  in Section 4. Estimation and   asymptotic properties of the new index are presented in Section 5, some applications are given in Section 6, Section 7   is conclusions and prospects.

 \numberwithin{equation}{section}
\section{ The new class of measures}\label{intro}

The following notation is used  throughout.
The Euclidean norm for random vectors $X = (X_1,\cdots,X_d)^T \in \Bbb{R}^d, d \ge 1$,
is $|X|_d=(x_1^2+\cdots+x_d)^{1/2}$ and $\langle u, v\rangle$ is the   standard Euclidean scalar product of $u$ and $v$. Whenever the dimension $d$ is clear in context, we omit the subscript.
The symbol $|\cdot|$ denotes the complex norm when its argument is complex; otherwise $|\cdot|$ denotes Euclidean norm. We denote by the imaginary unit $i=\sqrt{-1}$.
 Let
$\phi_{X,Y}(s,t)=E(e^{i\langle s,X\rangle+i\langle t,Y\rangle})$
 be the joint characteristic function of $(X,Y)\in \Bbb{R}^{p\times q},$  $\phi_{X}(s)= \phi_{X,Y}(s,0)$ and $\phi_{Y}(t)= \phi_{X,Y}(0,t)$ be the corresponding marginal characteristic functions.  For any
function $\phi$, we denote $|\phi|^2=\phi\overline{\phi}$, where $\overline{\phi}$ denote the
complex conjugate of $\phi$.

 \begin{definition} [Ebrahimi et al. (2010), Schennach (2019)] (i)  $N$ real valued random variables
 $X_1,\cdots, X_n$  are said to be sub-independent if $\phi_{\sum_{i=1}^n X_i}(t)=\prod_{i=1}^n \phi_{X_i}(t)$ for all $t\in \Bbb{R}$.

 (ii) Random  variables $X_1,\cdots, X_n$ are said to mutually sub-independent if for each subset $\{X_{11},\cdots, X_{ik}\}$ ($2\le k\le n$) of $\{X_1,\cdots, X_n\}$,
$\phi_{X_{i1}+\cdots+X_{ik}}(t)=\phi_{X_{i1},\cdots,X_{ik}}(t,\cdots,t)=\prod_{r=1}^k\phi_{X_{ir}}(t)$ for all $t\in \Bbb{R}$.
 \end{definition}

 Alternatively, in terms of distributions,  random variables $X_i, i=1,2,\cdots,n$   are  sub-independent, if the distribution
of $\sum_{i=1}^n X_i$ is the convolution of the distributions of $X_i$'s.
In this case, $X_i$'s are said to be a summable uncorrelated marginals (SUM). It is clear that the SUM and sub-independence are equivalent, so the two terminologies can be used
interchangeably (cf. Ebrahimi et al. (2010)).
The concept of sub-independence defined above can be extended to random vectors of arbitrary dimensions.

Although sub-independence is a weaker and more flexible assumption than full independence, to date, no  statistical  index has been available to measure this type of dependence. This paper fills that gap by proposing two new measures.

\begin{definition} Let $(X_i)_{i=1,...,n}$ be random  vectors with values in  $\Bbb{R}^p$.
   The  $\alpha$-sub-independent $(0<\alpha<2)$ covariance   between  random vectors   $X_1,\cdots, X_n$
   is the nonnegative number
$siCov^{(\alpha)}(X_1,\cdots, X_n)$ defined by
\begin{eqnarray}
 siCov^{(\alpha)}(X_1,\cdots, X_n)=\int_{\Bbb{R}^{p}}|\phi_{X_1,\cdots,X_n}(t,\cdots,t)-\prod_{i=1}^n\phi_{X_i}(t)|^2 \rho_{\alpha}(t)dt,
\end{eqnarray}
where
$\rho_{\alpha}(t)=c(p,\alpha)^{-1}|t|_p^{-\alpha-p}.$
The  $\alpha$-sub-independent correlation between random vectors   $X_1,\cdots, X_n$
  is the nonnegative number $siCor^{(\alpha)}(X_1,\cdots, X_n)$
defined by
\begin{eqnarray}
	siCor^{(\alpha)}(X_1,\cdots, X_n) = \frac{siCov^{(\alpha)}(X_1,\cdots, X_n)}{\int_{\Bbb{R}^{p}}\prod_{i=1}^n (1-|\phi_{X_i}(t)|^2)\rho_{\alpha}(t)dt },
\end{eqnarray}
  if   $\int_{\Bbb{R}^{p}}\prod_{i=1}^n (1-|\phi_{X_i}(t)|^2)\rho_{\alpha}(t)dt>0$, or 0 otherwise.	
\end{definition}

 For simplicity, we restrict our discussion to the case $n=2$ in what follows.
We call $1$-sub-independent   covariance as sub-independent   covariance,   $1$-sub-independent correlation as sub-independent correlation;   we write $siCov^{(\alpha)}(X)=siCov^{(\alpha)}(X, X)$, $siCov^{(\alpha)}(Y)=siCov^{(\alpha)}(Y, Y)$,
 $siCov(X,Y)=siCov^{(1)}(X,Y)$ and  $siCor(X,Y)=siCor^{(1)}(X,Y)$.

\begin{remark}
Condition $\int_{\Bbb{R}^{p}}(1-|\phi_X(t)|^2)(1-|\phi_Y(t)|^2)\rho_{\alpha}(t)dt>0$ is equivalent to
both $X$ and $Y$ are not  constant vectors. Unlike the normalized version of
distance covariance in Szekely et al. (2007),   $siCor^{(\alpha)}(X,Y)$  can not be defined in
the form of
 $$ siCor^{(\alpha)}(X,Y)=\frac{ siCov^{(\alpha)}(X, Y)}{\sqrt{ siCov^{(\alpha)}(X)\cdot siCov^{(\alpha)}(Y) }},$$
 which violates the usual condition that $siCor^{(\alpha)}(X,Y)\le 1$;   see a counterexample in    Remark 4.1.
\end{remark}

 For finiteness of  $siCov^{(\alpha)}(X, Y)$, it is sufficient that $E(|X|^{\alpha}_p+|Y|^{\alpha}_p)<\infty$.
The follow lemma is crucial which  can be found in Sz\'ekely and Rizzo (2009); see also von Bahr and Esseen  (1965) for the case $p=1$.
\begin{lemma} Assume that $0<\alpha<2$ and $X\in\Bbb{R}^p$, if $E|X|^{\alpha}_p<\infty$, then
\begin{eqnarray}
\int_{\Bbb{R}^p}\frac{1-Ee^{i\langle t,X\rangle}}{|t|_p^{\alpha+p}}dt=\int_{\Bbb{R}^p}\frac{1-E\cos\langle t,X\rangle}{|t|_p^{\alpha+p}}dt = c(p,\alpha) E|X|^{\alpha}_p.
\end{eqnarray}
\end{lemma}
To prove
$siCov^{(\alpha)}(X, Y)<\infty$,
by the Cauchy-Schwarz inequality, we have
\begin{eqnarray*}
|\phi_{X,Y}(t,t)-\phi_X(t)\phi_Y(t)|^2&=& |E(e^{i<t,X>}-\phi_X(t))(e^{i<t,Y>}-\phi_Y(t))|^2\\
&\le & \left(E[|e^{i<t,X>}-\phi_X(t)||e^{i<t,Y>}-\phi_Y(t)|]\right)^2\\
&\le& E|e^{i<t,X>}-\phi_X(t)|^2 E|e^{i<t,Y>}-\phi_Y(t)|^2\\
&=&(1-|\phi_X(t)|^2)(1-|\phi_Y(t)|^2).
\end{eqnarray*}
Then by Lemma 2.1 and by Fubini's theorem it follows that
\begin{eqnarray*}
siCov^{(\alpha)}(X, Y) &=&\int_{\Bbb{R}^{p}}|\phi_{X,Y}(t,t)-\phi_X(t)\phi_Y(t)|^2 \rho_{\alpha}(t)dt\\
&\le& \int_{\Bbb{R}^{p}}(1-|\phi_X(t)|^2)(1-|\phi_Y(t)|^2)\rho_{\alpha}(t)dt\\
&\le & \sqrt{\int_{\Bbb{R}^{p}}(1-|\phi_X(t)|^2)^2\rho_{\alpha}(t)dt \int_{\Bbb{R}^{p}}(1-|\phi_Y(t)|^2)^2\rho_{\alpha}(t)dt}\\
&\le & \sqrt{\int_{\Bbb{R}^{p}}(1-|\phi_X(t)|^2)\rho_{\alpha}(t)dt \int_{\Bbb{R}^{p}}(1-|\phi_Y(t)|^2)\rho_{\alpha}(t)dt}\\
&=&\sqrt{ E|X-X'|^{\alpha}_p E|Y-Y'|^{\alpha}_p}<\infty.
\end{eqnarray*}
Similarly,
\begin{eqnarray*}
 siCov^{(\alpha)}(X)&=&\int_{\Bbb{R}^{p}}|\phi_{X,X}(t,t)-\phi_X(t)\phi_X(t)|^2 \rho_{\alpha}(t)dt\\
&\le& \int_{\Bbb{R}^{p}} (E|e^{i<t,X>}-\phi_X(t)|^2)^2\rho_{\alpha}(t)dt\\
&=& \int_{\Bbb{R}^{p}} \left(E(e^{i<t,X>}-\phi_X(t))(e^{-i<t,X>}-\phi_X(-t)))\right)^2\rho_{\alpha}(t)dt\\
&=& \int_{\Bbb{R}^{p}}(1-Ee^{i<t,X>}Ee^{i<-t,X>})^2\rho_{\alpha}(t)dt\\
&= & \int_{\Bbb{R}^{p}}(1-|\phi_X(t)|^2)^2\rho_{\alpha}(t)dt\\
&\le & \int_{\Bbb{R}^{p}}(1-|\phi_X(t)|^2)\rho_{\alpha}(t)dt\\
&=& E|X-X'|^{\alpha}_p<\infty,
\end{eqnarray*}
and
\begin{eqnarray*}
siCov^{(\alpha)}(Y)=\int_{\Bbb{R}^{p}}|\phi_{Y,Y}(t,t)-\phi_Y(t)\phi_Y(t)|^2 \rho_{\alpha}(t)dt\le E|Y-Y'|^{\alpha}_p<\infty.
\end{eqnarray*}

 The following results summarize the key properties of sub-independent covariance.
\begin{theorem} The following properties hold:

(1). For all $X, Y \in \Bbb{R}^p$,  $ siCov^{(\alpha)}(X, Y)\ge 0$, and  $ siCov^{(\alpha)}(X, Y)= 0$ if and only if $X$ and $Y$ are sub-independent.

(2). For all $X \in \Bbb{R}^p$, $ siCov^{(\alpha)}(X)=0$ iff  $X = E(X)$, almost surely.

(3). $ siCov^{(\alpha)}(a_1+bCX, a_2+bCY)=|b|siCov^{(\alpha)}(X, Y)$, for all constant vectors $a_1,a_2\in\Bbb{R}^p$,   scalar $b$ and   $p\times p$ orthonormal matrix $C$.

(4). $ siCov^{(\alpha)}(a+bCX)=|b| siCov^{(\alpha)}(X)$, for all constant vector $a\in\Bbb{R}^p$,   scalar $b$ and orthonormal matrix $C\in\Bbb{R}^p$.

(5). If random variables $X$ and $Y$ are  sub-independent, then for any $\alpha>0$ such that $E(|X|^{\alpha}+|Y|^{\alpha})<\infty$, one has $ siCov^{(\alpha)}(X, Y)=0$.  Conversely,
if $ siCov^{(\alpha)}(X, Y)=0$ for some $\alpha>0$, then $X$ and $Y$ are  sub-independent.
\end{theorem}
{\bf Proof}. (1) If $X$ and $Y$  are sub-independent then  $\phi_{X,Y}(t,t)= \phi_{X}(t)\phi_{Y}(t)$ for
all $t\in \Bbb{R}^p$, thus $ siCov^{(\alpha)}(X, Y)= 0$. Conversely, if  $X$ and $Y$ are  not sub-independent, then there exists an open set $B\subset\Bbb{R}^p$ such that
$|\phi_{X,Y}(t,t)-\phi_X(t)\phi_Y(t)|^2>0$ for all $t\in B$, since
characteristic functions are continuous. Hence, $ siCov^{(\alpha)}(X, Y)>0$.\\
(2)  $ siCov^{(\alpha)}(X)=0$ iff  $\phi_{X,X}(t,t)= \phi_{X}(t)\phi_{X}(t)$ for
all $t\in \Bbb{R}^p$, hence $X$ is a constant vector $EX$, almost surely.\\
(3) For all constant vectors $a_1,a_2\in\Bbb{R}^p$,   scalar $b$ and $p\times p$ orthonormal matrix $C$, we have
 \begin{eqnarray*}
   &&siCov^{(\alpha)}(a_1+bCX, a_2+bCY)\\
   &=&\int_{\Bbb{R}^{p}}|Ee^{it'(bCX+bCY)}- Ee^{it'(bCX)} Ee^{it'(bCY)}|^2 \rho_{\alpha}(t)dt\\
  &=& \int_{\Bbb{R}^{p}}|Ee^{i(bC't)'(X+Y)}- Ee^{i(bC't)'X} Ee^{i(bC't)'Y}|^2 \rho_{\alpha}(t)dt\\
  &=& |b|\int_{\Bbb{R}^{p}}|Ee^{is'(X+Y)}- Ee^{is'X} Ee^{is'Y}|^2 \rho_{\alpha}(s)ds\\
  &=&|b| siCov^{(\alpha)}(X, Y),
 \end{eqnarray*}
 where in the last equality we have used the transformation $s=bC't$.\\
 (4) This is a special case of (3) for   $a_1=a_2$ and $Y=X$. Statement (5) is obvious. $\hfill\square$

Analogous properties to the product-moment correlation coefficient for  sub-independent correlation  are established in Theorem 2.2.

\begin{theorem} If $X, Y$  are  $\Bbb{R}^p$-valued random variables, and $E|X|^{\alpha}_p+E|Y|^{\alpha}_p<\infty$, then,  the following properties hold:

 (i) $0\le siCor^{(\alpha)}(X,Y)\le 1$;

 (ii) If $siCor^{(\alpha)}(X,Y)= 0$ for some $\alpha$, then  $X$ and $Y$ are sub-independent; Conversely, if   $X$ and $Y$ are  sub-independent, then for any $\alpha>0$ we have
  $siCov^{(\alpha)}(X, Y)=0$.

 (iii)   $siCor^{(\alpha)}(a_1+bCX, a_2+bCY)=1$ for all constant vectors $a_1,a_2\in\Bbb{R}^p$,   scalar $b$ and   $p\times p$ orthonormal matrix $C$.
\end{theorem}

{\bf Proof}. The statement $ siCor^{(\alpha)}(X,Y)\ge 0$ is obvious. The only thing to show is that
  $ siCor^{(\alpha)}(X,Y)\le 1$.  We just note that
\begin{eqnarray*}
 siCov^{(\alpha)}(X,Y)&=&\int_{\Bbb{R}^{p}}|\phi_{X,Y}(t,t)-\phi_X(t)\phi_Y(t)|^2 \rho_{\alpha}(t)dt\\
&\le& \int_{\Bbb{R}^{p}}(1-|\phi_X(t)|^2)(1-|\phi_Y(t)|^2)\rho_{\alpha}(t)dt,
\end{eqnarray*}
this proves (i). Statement (ii) is obvious.

(iii)   This follows from   Theorem 2.1 (iii) and (iv). $\hfill\square$

\begin{theorem} Suppose that $(X_1, Y_1),\cdots, (X_4, Y_4)$ are independent copies of $(X, Y)\in \Bbb{R}^{p\times p}$.
If $E|X|^{\alpha}_p+E|Y|^{\alpha}_p<\infty$, then
\begin{eqnarray}
  siCov^{(\alpha)}(X, Y) &=& 2E|X_1+Y_1-X_2-Y_3|^{\alpha}_p-E|X_1+Y_2-X_3-Y_4|^{\alpha}_p\nonumber\\
&&-E|X_1+Y_1-X_2-Y_2|^{\alpha}_p,
\end{eqnarray}
\begin{eqnarray*}
  siCov^{(\alpha)}(X) &=& 2E|2X_1-X_2-X_3|^{\alpha}_p-E|X_1+X_2-X_3-X_4|^{\alpha}_p\\
&&-2E|X_1-X_2|^{\alpha}_p,
\end{eqnarray*}
and
\begin{eqnarray*}
  siCov^{(\alpha)}(Y) &=& 2E|2Y_1-Y_2-Y_3|^{\alpha}_p-E|Y_1+Y_2-Y_3-Y_4|^{\alpha}_p\\
&&-2E|Y_1-Y_2|^{\alpha}_p.
\end{eqnarray*}
\end{theorem}
{\bf Proof}. We provide the proof for (2.4), the proofs of other parts are  similarly.
 Clearly,
\begin{eqnarray*}
\phi_{X,Y}(t,t)\overline{\phi_{X,Y}(t,t)}&=&E e^{it'(X+Y)}E e^{-it'(X+Y)}\\
&=&E e^{it'(X+Y)}E e^{-it'(X_2+Y_2)}\\
&=& E e^{it'(X_1+Y_1-X_2-Y_2)},
\end{eqnarray*}
\begin{eqnarray*}
\phi_{X,Y}(t,t)\overline{\phi_X(t)\phi_Y(t)} &=&E e^{it'(X+Y)} E e^{-it'(X_2+Y_3)}\\
&=&E e^{it'(X_1+Y_1-X_2-Y_3)},
\end{eqnarray*}
\begin{eqnarray*}
\overline{\phi_{X,Y}(t,t)}\phi_X(t)\phi_Y(t)&=&E e^{-it'(X+Y)} E e^{it'(X_2+Y_3)}\\
&=&E e^{-it'(X_1+Y_1-X_2-Y_3)},
\end{eqnarray*}
\begin{eqnarray*}
\phi_X(t)\phi_Y(t)\overline{\phi_X(t)\phi_Y(t)}& =&E e^{it'(X_1+Y_2)} E e^{-it'(X_3+Y_4)}\\
&=&E e^{it'(X_1+Y_2-X_3-Y_4)}.
\end{eqnarray*}
Then, by applying Lemma 2.1, we obtain
\begin{eqnarray*}
  siCov^{(\alpha)}(X, Y)&=&\int_{\Bbb{R}^p}(\phi_{X,Y}(t,t)-\phi_X(t)\phi_Y(t))\overline{(\phi_{X,Y}(t,t)-\phi_X(t)\phi_Y(t))} \rho_{\alpha}(t)dt\\
 &=&-\int_{\Bbb{R}^p}[\phi_{X,Y}(t,t)\overline{\phi_X(t)\phi_Y(t)}+\overline{\phi_{X,Y}(t,t)}\phi_X(t)\phi_Y(t)] \rho_{\alpha}(t)dt\\ &&+\int_{\Bbb{R}}[\phi_X(t)\phi_Y(t)\overline{\phi_X(t)\phi_Y(t)}+\phi_{X,Y}(t,t)\overline{\phi_{X,Y}(t,t)}] \rho_{\alpha}(t)dt\\
&=&\int_{\Bbb{R}^p}\left(1-\phi_{X,Y}(t,t)\overline{\phi_X(t)\phi_Y(t)}\right)\rho_{\alpha}(t)dt\\
&&+\int_{\Bbb{R}^p}\left(1-\overline{\phi_{X,Y}(t,t)}\phi_X(t)\phi_Y(t)\right) \rho_{\alpha}(t)dt\\
&&-\int_{\Bbb{R}^p}\left(1-\phi_X(t)\phi_Y(t)\overline{\phi_X(t)\phi_Y(t)}\right)\rho_{\alpha}(t)dt\\
&&-\int_{\Bbb{R}^p}\left(1-\phi_{X,Y}(t,t)\overline{\phi_{X,Y}(t,t)}\right) \rho_{\alpha}(t)dt\\
&=&\int_{\Bbb{R}^p}\left(1- E e^{it'(X_1+Y_1-X_2-Y_3)}\right)\rho_{\alpha}(t)dt\\
&&+\int_{\Bbb{R}^p}\left(1- E e^{-it'(X_1+Y_1-X_2-Y_3)}\right) \rho_{\alpha}(t)dt\\
&&-\int_{\Bbb{R}^p}\left(1- E e^{it'(X_1+Y_2-X_3-Y_4)}\right)\rho_{\alpha}(t)dt\\
&&-\int_{\Bbb{R}^p}\left(1- E e^{it'(X_1+Y_1-X_2-Y_2)}\right) \rho_{\alpha}(t)dt\\
&=& 2E|X_1+Y_1-X_2-Y_3|^{\alpha}_p-E|X_1+Y_2-X_3-Y_4|^{\alpha}_p\\
&&-E|X_1+Y_1-X_2-Y_2|^{\alpha}_p.
\end{eqnarray*}
This completes the proof.  $\hfill\square$

The next theorem gives a equivalent representation of  $siCor(X,Y)$.

\begin{theorem}  Suppose that $(X_1, Y_1),\cdots, (X_4, Y_4)$ are independent, identically
distributed  copies of $(X, Y)\in \Bbb{R}^{p\times p}$,   if $E(|X|^{\alpha}_p+|Y|^{\alpha}_p)<\infty$, then
\begin{eqnarray}
siCor^{(\alpha)}(X,Y)=\frac{2{\cal{J}}_1(X,Y)-{\cal{J}}_2(X,Y)-{\cal{J}}_3(X,Y) }{{\cal{K}}_1(X,Y)+{\cal{K}}_2(X,Y)-{\cal{K}}_3(X,Y) },
\end{eqnarray}
where
\begin{eqnarray*}
{\cal{J}}_1(X,Y)&=&E|X_1+Y_1-X_2-Y_3|^{\alpha}_p,\\
{\cal{J}}_2(X,Y)&=&E|X_1+Y_2-X_3-Y_4|^{\alpha}_p,\\
{\cal{J}}_3(X,Y)&=&E|X_1+Y_1-X_2-Y_2|^{\alpha}_p,
\end{eqnarray*}
and
\begin{eqnarray*}
{\cal{K}}_1(X,Y)&=&E|X_1-X_2|^{\alpha}_p,\\
{\cal{K}}_2(X,Y)&=&E|Y_3-Y_4|^{\alpha}_p, \\
{\cal{K}}_3(X,Y)&=&E|X_1-X_2+X_3-Y_4|^{\alpha}_p.
\end{eqnarray*}
\end{theorem}
{\bf Proof}.  ${\cal{J}}_1(X,Y),
{\cal{J}}_2(X,Y),$
and  ${\cal{J}}_3(X,Y)$  can be obtained by (2.4).  For the convenience of writing, setting $\triangle=\int_{\Bbb{R}^{p}}(1-|\phi_X(t)|^2)(1-|\phi_Y(t)|^2)\rho_{\alpha}(t)dt$, then
\begin{eqnarray*}
 \triangle&=&\int_{\Bbb{R}^{p}}(1-e^{it(X_1-X_2)})(1- e^{it(Y_3-Y_4)})\rho_{\alpha}(t)dt\\
&=&\int_{\Bbb{R}^{p}}(1-e^{it(X_1-X_2)}) \rho_{\alpha}(t)dt\\
&&+\int_{\Bbb{R}^{p}}(1-e^{it(Y_3-Y_4)}) \rho_{\alpha}(t)dt\\
&&- \int_{\Bbb{R}^{p}}(1-e^{it(X_1-X_2+Y_3-Y_4)}) \rho_{\alpha}(t)dt\\
&=& E|X_1-X_2|^{\alpha}_p+E|Y_3-Y_4|^{\alpha}_p -E|X_1-X_2+X_3-Y_4|^{\alpha}_p\\
&=& {\cal{K}}_1(X,Y)+{\cal{K}}_2(X,Y)-{\cal{K}}_3(X,Y).
\end{eqnarray*}
Hence, the theorem.  $\hfill\square$

\section{Comparisons    with known dependencies}\label{intro}

The classical Pearson product-moment correlations between two variables have been generalized to measure association between two   variables  in many ways. For example, the
distance covariance (dCov) coefficient, martingale difference correlation (Shao and Zhang (2014)) and sub-independent coefficient.  A large number of literature uses these coefficients to test whether two random vectors are correlated with each other. If so, it is important to reveal the patterns present in these associations. We  discuss the relationship among  Pearson product-moment correlation, distance covariance coefficient,  sub-independent coefficient and  Positive Orthant Dependence (POD) and
Negative Orthant Dependence (NOD)    dependence. For details about POD (NOD) and
other notions of dependence, see, e.g.,   Barlow and Proschan (1975).

Ebrahimi et al. (2010) pointed out that POD (NOD) are the weakest among all existing notions of dependence. We   remark that the sub-independence and POD (NOD) there are no logical implication relationship. To see that, let $X=(X_1,\cdots,X_n)$ be a nonnegative random vector in $\Bbb{R}^n$, then, for $t\ge 0$,
\begin{eqnarray*}
\phi_{\sum_{k=1}^n X_{k}}(t)&-&\prod_{i=1}^n\phi_{X_{i}}(t)\\
&=&\int_0^{\infty}\cdots \int_0^{\infty}e^{-t\sum_{i=1}^{n}x_i}\left(\bar{F}(x_1,\cdots,x_n)-\prod_{i=1}^n\bar{F_i}(x_i)\right)dx_1\cdots x_n,
\end{eqnarray*}
from which we  cannot be reduced that there is a mutual inclusion relationship between  POD and sub-independence. Namely, there exist random variable pairs that satisfy POD but do not satisfy sub-independence (for example, variables that are strongly positively correlated but the distribution of their sum is not in a convolutional form).
        There are also cases where sub-independence is satisfied but POD is not (for example, the distribution of the sum conforms to convolution, but the tail exhibits negative dependence). In the SUM distribution class, POD+sub-independence $\Rightarrow$ independence.

For rvs $U$and $V$,  Shao and Zhang (2014) proposed a martingale
difference correlation for high-dimensional variable screening,  they called $V-E(V)$ a martingale
difference with respect to $U$ if $E(V|U)=E(V)$. This  relationship lies in between independence and uncorrelatedness, and neither of this relation and sub-independence implies the other one. Namely, The  martingale difference correlation and sub-independence  have no inclusion relationship in the mathematical framework.

Edelmann et al. (2020) asserted that ``$|{\rm dCor}(X,Y)|=1$ if and only if   $Y$ is a linear function of $X$, almost surely".
 Edelmann et al. (2021)also asserted that ``$|{\rm Cor}(X,Y)|=1$ if and only if   ${\rm dCor}^2(X,Y) =1$".



 The following result shows that  the  sub-independence lies in between independence and uncorrelatedness.

\begin{theorem} If $X, Y$  are  real-valued random variables, and $EX^2+EY^2<\infty$, then

 (i) ${\rm dCor}^2(X,Y) =0 \Rightarrow siCor(X,Y)=0 \Rightarrow  {\rm Cor}(X,Y)=0,$ but the converse is not true.

 (ii)   ${\rm Cor}(X,Y) \neq 0 \Rightarrow siCor(X,Y) \neq 0 \Rightarrow  {\rm dCor}^2(X,Y)\neq 0$.

  (iii) If $Y=X+a$ for constant  $a$, then  $siCor (X,Y)=1$.
\end{theorem}
{\bf Proof}. (i) The proof is straightforward. In fact, ${\rm dCor}^2(X,Y) =0  \Leftrightarrow  \phi_{X,Y}(s,t)= \phi_{X}(s)\phi_{Y}(t)$ for
all $s,t\in \Bbb{R}$ $\Rightarrow   \phi_{X,Y}(t,t)= \phi_{X}(t)\phi_{Y}(t)$ for
all $t\in \Bbb{R}$ $\Rightarrow E(XY)=E(X)E(Y)$ $ \Leftrightarrow  {\rm Cor}(X,Y)=0$.
The example  in the next section illustrates the converse of (i) is not true.

It is not difficult
to find examples from standard textbooks that there are random
vectors such that  ${\rm Cor}(X,Y)=0$ holds but   $siCor(X,Y)=0$ is violated, or  $siCor(X,Y)=0$  holds but  ${\rm dCor}^2(X,Y) =0$
is not satisfied.

Statement (ii) is easily obtained by (i).  $\hfill\square$

\section{Two examples}

This section provides $siCov (X, Y)$ and $siCor (X, Y)$ for  two-dimensional normal distribution and two-dimensional  Cauchy distribution.

\subsection{ Results for the bivariate normal distribution  }\label{intro}

Let $(X,Y)$  have    the bivariate standard normal distribution $N_2(0,0,1,1,\rho)$.  Suppose that $(X_1, Y_1),\cdots, (X_4, Y_4)$ are independent, identically
distributed  copies of $(X, Y)$,  then
\begin{eqnarray*}
&&X_1+Y_1-X_2-Y_3\sim N(0, 4+2\rho),\,  X_1+Y_2-X_3-Y_4\sim N(0, 4),\\
&&X_1+Y_1-X_2-Y_2\sim N(0, 4(1+\rho)),
\end{eqnarray*}
and hence
\begin{eqnarray*}
&&E|X_1+Y_1-X_2-Y_3|=\sqrt{4+2\rho}\sqrt{\frac{2}{\pi}}, \,  E|X_1+Y_2-X_3-Y_4|=2\sqrt{\frac{2}{\pi}},\\
&& E|X_1+Y_1-X_2-Y_2|=2\sqrt{1+\rho}\sqrt{\frac{2}{\pi}}.
\end{eqnarray*}
By making use of Theorem 2.3 yields
\begin{eqnarray*}
 siCov(X, Y)=2\sqrt{4+2\rho}\sqrt{\frac{2}{\pi}}-2\sqrt{\frac{2}{\pi}}
-2\sqrt{1+\rho}\sqrt{\frac{2}{\pi}}.
\end{eqnarray*}
Similarly,
\begin{eqnarray*}
  siCov(X, X)=  siCov(Y, Y) =2\sqrt{\frac{2}{\pi}}(\sqrt{6}-\sqrt{2}-1),
\end{eqnarray*}
and $${\cal{K}}:=E|X_1-X_2|+E|Y_3-Y_4| -E|X_1-X_2+X_3-Y_4|=2\sqrt{\frac{2}{\pi}}(\sqrt{2}-1). $$
Finally, by (2.5), we have
\begin{eqnarray*}
siCor_{\rho}(X,Y):= siCor(X,Y) =\frac{\sqrt{4+2\rho}-\sqrt{1+\rho}-1}{\sqrt{2}-1},
\end{eqnarray*}
from which we get  $siCor(X,Y)\le |\rho|$, and
\begin{eqnarray*}
siCor_{0}(X,Y)=0,\, siCor_{-1}(X,Y)=1,\,  siCor_{1}(X,Y)= \frac{\sqrt{6}-\sqrt{2}-1}{\sqrt{2}-1}\doteq 0.085.
\end{eqnarray*}
\begin{remark} Define
$${\cal{R}_{\rho}}:=\frac{ siCov(X, Y)}{\sqrt{ siCov(X, X)\cdot siCov(Y, Y) }},  $$
then
$${\cal{R}_{\rho}} =\frac{\sqrt{4+2\rho}-\sqrt{1+\rho}-1}{\sqrt{6}-\sqrt{2}-1},  $$
In particular,
$${\cal{R}}_{0}=0, \, {\cal{R}}_{-1}=\frac{\sqrt{2}-1}{\sqrt{6}-\sqrt{2}-1}\doteq 11.74, \, {\cal{R}}_{1}=1. $$
  It turns out that  ${\cal{R}_{\rho}}$ can not satisfied the desire condition $ {\cal{R}_{\rho}}\le 1$.
\end{remark}

\subsection{ Results for the bivariate Cauchy distribution  }\label{intro}

The bivariate Cauchy distribution is a continuous probability distribution used to describe the joint behavior of two random variables, and is an important special case of multivariate heavy tailed distributions. It has special significance in statistical modeling, especially suitable for scenarios with strong correlation between extreme events and where traditional normality assumptions do not hold.
The $2$-dimensional random vector ${\bf X}=(X_1, X_2)'$ is said to have a multivariate Cauchy distribution (denoted by ${\bf X}\sim MC_2({\boldsymbol \mu},{\bf \Sigma})$)
if its probability density function is given by (see Fang, Kotz, and Ng 1990)
\begin{equation*}
\frac{1}{2\pi}|{\bf \Sigma}|^{-\frac12}\left(1+({\bf x}-{\boldsymbol \mu})'|{\bf \Sigma}|^{-1}({\bf x}-{\boldsymbol \mu}\right)^{-\frac{3}{2}}, \, {\bf x}\in \Bbb{R}^2,
\end{equation*}
where ${\boldsymbol \mu} \in \Bbb{R}^2$   is its location parameter and  ${\bf \Sigma}$ is  $2\times 2$  positive semi-definite matrix.
The characteristic function of ${\bf X}\sim MC_2({\boldsymbol \mu},{\bf \Sigma})$ is given by
\begin{equation*}
\phi_{\bf X}({\bf t})=e^{i{\bf t}'{\boldsymbol \mu}}\exp(-\sqrt{{\bf t}'{\bf \Sigma}{\bf t}}),\; {\bf t}\in \Bbb{R}^2.
\end{equation*}
In particular, the standard multivariate Cauchy distribution $MC_2({\bf 0}, {\bf I_2})$  has the density
\begin{equation*}
\frac{1}{2\pi}\left(1+{\bf x}'{\bf x}\right)^{-\frac{3}{2}},\, {\bf x}\in \Bbb{R}^2,
\end{equation*}
and the characteristic function is given by
$$\phi_{(X_1,X_2)}(t_1,t_2)=\exp\left(-\sqrt{t_1^2+t_2^2}\right), \; -\infty<t_1, t_2<\infty,$$
 It follows that
$$\phi_{X_1}(t)=\phi_{X_2}(t)= \exp\left(-|t|\right), \; -\infty<t<\infty,$$
and
$$\phi_{X_1+X_2}(t)=\exp\left(-\sqrt{2}|t|\right), \; -\infty<t<\infty.$$
By (2.1), for $0<\alpha<1$,
\begin{eqnarray*}
 siCov^{(\alpha)}(X_1, X_2)&=&\frac{1}{\pi}\int_{\Bbb{R}}|\exp\left(-\sqrt{2}|t|\right) - \exp\left(-2|t|\right)|^2  |t|^{-1-\alpha}dt\\
 &=&\frac{2}{\pi}\int_{0}^{\infty}|\exp\left(-\sqrt{2}t\right) - \exp\left(-2t\right)|^2  t^{-1-\alpha}dt\\
  &=&\frac{2}{\pi}\int_{0}^{\infty}(e^{-4t} -2e^{-(2+\sqrt{2})t}+e^{-2\sqrt{2}t})  t^{-1-\alpha}dt\\
  &=&\frac{2}{\pi\alpha}\left(-4^{\alpha}+2(2+\sqrt{2})^{\alpha}-(2\sqrt{2})^{\alpha}\right)\Gamma(1-\alpha),
\end{eqnarray*}
and
\begin{eqnarray*}
\int_{\Bbb{R}}(1-|\phi_{X_1}(t)|^2)(1-|\phi_{X_2}(t)|^2)\rho_{\alpha}(t)dt&=&\frac{1}{\pi}\int_{\Bbb{R}}(1-e^{-2|t|})^2 |t|^{-1-\alpha}dt\\
&=&\frac{2}{\pi}\int_{0}^{\infty}(1-e^{-2t})^2 t^{-1-\alpha}dt\\
&=&\frac{8}{\pi\alpha}(2^{\alpha-1}-4^{\alpha-1})\Gamma(1-\alpha).
\end{eqnarray*}
Finally, by (2.2) we obtain
\begin{eqnarray}
	siCor^{(\alpha)}(X,Y) = \frac{-4^{\alpha}+2(2+\sqrt{2})^{\alpha}-(2\sqrt{2})^{\alpha}}{4(2^{\alpha-1}-4^{\alpha-1})}.
\end{eqnarray}
It can be proven that  $siCor^{(\alpha)}$ is monotonically increasing with respect to $\alpha$, and $0< siCor^{(\alpha)}(X,Y)<0.04$.

\section{Estimation and   asymptotic properties  }\label{intro}

In this section, we  give   an unbiased estimate of $siCov^{(\alpha)}(X, Y)$ (denote by $\widehat{siCov^{(\alpha)}}_n(X, Y)$), and derive the asymptotic distribution of $\widehat{siCov^{(\alpha)}}_n(X, Y)$.

Define the symmetric kernel function
\begin{eqnarray*}
k((X_1,Y_1),(X_2,Y_2),(X_3,Y_3),(X_4,Y_4))&=&\frac{1}{12}\sum_{\stackrel{1\le i, j, k\le 4}{i\neq j\neq k}}  |X_i+Y_i-X_j-Y_k|^{\alpha}_p\\
&&-\frac{1}{24}\sum_{\stackrel{1\le i, j, k, l\le 4}{i\neq j\neq k\neq l}}|X_i+Y_j-X_k-Y_l|^{\alpha}_p\\
&&-\frac{1}{24}\sum_{\stackrel{1\le i, j\le 4}{i\neq j}}|X_i+Y_i-X_j-Y_j|^{\alpha}_p.
\end{eqnarray*}
Given $n$ samples $(X_1, Y_1),\cdots, (X_n, Y_n)$ from $(X,Y)$, an  estimator of    $siCov^{(\alpha)}(X,Y)$   can be defined as
\begin{eqnarray*}
 \widehat{siCov^{(\alpha)}}_n(X, Y) =\frac{1}{C_n^4} \sum_{1\le i<j<k<l\le n} k\left((X_i,Y_i),(X_j,Y_j),(X_k,Y_k),(X_l,Y_l)\right).
\end{eqnarray*}
We directly obtain $E(\widehat{siCov^{(\alpha)}}_n(X, Y)) =siCov^{(\alpha)}(X, Y)$. That is  $\widehat{siCov^{(\alpha)}}_n(X, Y)$ is an unbiased estimate of $siCov^{(\alpha)}(X, Y)$.

Using the general theory of  $U$-statistics (see, e.g.  Serfling (1980, P. 190, Theorem A)) we have the following  result.

\begin{theorem}  Suppose that  $E|X|^{\alpha}_p+E|Y|^{\alpha}_p<\infty$. Given $n$ samples $(X_1, Y_1),$ $\cdots, (X_n, Y_n)$ from $(X,Y)$, then,
$$ \widehat{siCov^{(\alpha)}}_n(X, Y) \stackrel{a.s.}{\longrightarrow} siCov^{(\alpha)}(X, Y), \, as\, n\to\infty,$$
and
$$ \widehat{siCor^{(\alpha)}}_n(X, Y) \stackrel{a.s.}{\longrightarrow} siCor^{(\alpha)}(X, Y), \, as\, n\to\infty.$$
\end{theorem}

The following theorem  gives the asymptotic distribution of the test statistic
based on a known result on the convergence of $U$-statistics
(Serfling (1980, Chapter 5.5.1 Theorem A)). We state the result without a proof.

\begin{theorem}  Suppose that $0<Var(k_1(X,Y))<\infty$. Given $n$ samples $(X_1, Y_1),\cdots, (X_n, Y_n)$ from $(X,Y)$, then,
$${n}^{\frac 12}\left(\widehat{siCov^{(\alpha)}}_n(X, Y)-siCov^{(\alpha)}(X, Y)\right)\stackrel{\cal{D}}{\longrightarrow} N(0, 16Var((k_1(X,Y)))),$$
  as $ n\to\infty,$
where
$$k_1((x_1,y_1))=E_{2,3,4} k\left((x_1,y_1),(X_2,Y_2),(X_3,Y_3),(X_4,Y_4)\right).$$
Here, $E_{2,3,4}$ stands for taking expectation over $(X_2,Y_2),(X_3,Y_3),(X_4,Y_4)$.
\end{theorem}

If $0<Var(k_1(X,Y))<\infty$, then $X$ and $Y$ cannot be  sub-independent, otherwise,    $k_1(X,Y)=0$ almost surely.  When $X$ and $Y$ are sub-independent, the asymptotic distribution of
$${n}^{\frac12}\left(\widehat{siCov^{(\alpha)}}_n(X, Y)-siCov^{(\alpha)}(X, Y)\right)$$  is no
longer normal. The asymptotic distribution can be
derived using results on first-order degenerate $U$-statistics.

\begin{theorem}  Suppose $X$ and $Y$ are sub-independent with $E|X|^{\alpha}_p+E|Y|^{\alpha}_p<\infty$. Given $n$ samples $(X_1, Y_1),\cdots, (X_n, Y_n)$ from $(X,Y)$, then,
$${n}^{\frac12} \widehat{siCov^{(\alpha)}}_n(X, Y)  \stackrel{\cal{D}}{\longrightarrow} 6\sum_{i=1}^{\infty} \lambda_i(\chi_i^2-1), \, as\, n\to\infty,$$
where $\chi_i^2$'s  are i.i.d. $\chi^2(1)$ variates and
$\{\lambda_i\}$  are the eigenvalues of  operator $G$ which is defined as
$Gf(x_1,y_1)=E\left(k_{12}(x_1,y_1), (X_2,Y_2)f(X_2,Y_2)\right),$
where
$$k_{12}((x_1,y_1), (x_2,y_2))=E_{3,4} k\left((x_1,y_1),(x_2,y_2),(X_3,Y_3),(X_4,Y_4)\right).$$
Here,   $E_{3,4}$ stands for taking expectation over $(X_3,Y_3),(X_4,Y_4)$.
\end{theorem}
{\bf Proof}. The asymptotic distribution of $\widehat{siCov^{(\alpha)}}_n(X, Y)$ can be proved by modifying the proof of Theorem  of Serfling
(1980, p. 194).
$\hfill\square$

\section{Applications }\label{intro}

Many well-known results in probability and statistics
are often based on the distribution of the sums of independent (and often
identically distributed) random variables rather than the joint distribution of the
summands. In other words, it is the convolution of the marginal distributions which is needed rather than the joint
distribution of the summands, which in the case of independence.
 Therefore, the full strength of independence of the summands will not
be required. The concept of
sub-independence can replace that of independence in most of the theorems   in
these situations which deal with the distribution of the sum of random
variables. For example,  in Khintchine's law of large numbers and Lindeberg-Levy's
central limit theorem as well as other important theorems in probability and
statistics (cf.  summability of random variables, compound of Poisson processes,  renewal theorems, and so on),  we have only used the  characteristic functions of the sum of random variables $X_i$'s. More precisely, the proof does not employ the full power of independence of  $X_i$'s, it only uses the
fact that $\phi_{X_{1}+\cdots+X_{n}}(t) =\prod_{i=1}^n\phi_{X_{i}}(t)$ for all $t\in \Bbb{R}$.
Therefore, the assumption
of  independent and identically distributed throughout this paper can be replaced with that of  sub-independent  and identically distributed whenever appropriate.
For more details, see Hamedani and Maadooliat (2015).  We mention below a few results in which the assumption of independence is replaced by
that of sub-independence.


\subsection{Concentration inequalities in machine learning}\label{intro}

Concentration inequalities give probability bounds for a random variable to be concentrated
around its mean, or for it to deviate from its mean or some other value. We present several concentration inequalities used in machine learning.

We start this subsection with the  sub-independent version of  Hoeffding's inequality (Hoeffding (1963, Theorem
2), see also  Theorem D.2 in  Mohri et al. (2018).

\begin{theorem} [Hoeffding's inequality] Let $X_1,\cdots, X_n$  be sub-independent random variables with
$X_i$ taking values in $[a_i, b_i]$ $(i=1,2,\cdots,n)$. Then, for any $\varepsilon>0$, the following inequalities hold for $S_n=\sum_{i=1}^n X_i$:
$$P(S_n-ES_n\ge \varepsilon)\le e^{-2\epsilon^2}/\sum_{i=1}^n(b_i-a_i)^2,$$
and
$$P(S_n-ES_n\le -\varepsilon)\le e^{-2\epsilon^2}/\sum_{i=1}^n(b_i-a_i)^2.$$
\end{theorem}
{\bf Proof}. We only note that, for $t>0$,  $Ee^{tS_n}=\prod_{i=1}^n Ee^{tX_i}$.  The rest follows from the proof given in Mohri et al. (2018).
$\hfill\square$

 The following theorem presents a finer upper bound than Hoeffding's inequality expressed in terms of the binary relative entropy.
\begin{theorem} [Sanov's theorem] Let $X_1,\cdots, X_n$  be sub-independent random variables drawn according to some distribution $F$ with mean $p$ and support included
in $[0, 1]$. Then, for any $q\in [0,1]$, the following inequality holds for $\hat{p}=1/n\sum_{i=1}^n X_i$:
$$P(\hat{p}\ge q)\le e^{-n D(q||p)},$$
where $D(q||p)=q\log\frac{q}{p}+(1-q)\log \frac{1-q}{1-p}$   is the binary relative entropy of $p$ and $q$.
\end{theorem}

Sanov's theorem,  is based on the concept of sub-independence, can be used to prove the following multiplicative Chernoff bounds.
which strengthen the characterizations given in  Mohri et al. (2018, Theorem D.4)  under the assumption of independence.

\begin{theorem} [Multiplicative Chernoff bounds] Let $X_1,\cdots, X_n$  be sub-independent random variables drawn according to some distribution $F$ with mean $p$ and support included
in $[0, 1]$. Then, for any $\gamma\in [0,1/p-1]$, the following inequalities hold for $\hat{p}=1/n\sum_{i=1}^n X_i$:
$$P(\hat{p}\ge (1+\gamma)p)\le e^{-\frac{np\gamma^2}{3}},$$
and
$$P(\hat{p}\le (1-\gamma)p)\le e^{-\frac{np\gamma^2}{2}}.$$
\end{theorem}



\subsection{The  sub-independent version of actuarial models}

In this section, we present some important models and settings in actuarial science based on the concept
of sub-independence,  where it has traditionally been assumed that independence holds between certain risks and assume that  the
number of terms is independent of the summands  (see, e.g. Kaas et al. (2008)).

\subsubsection{The individual risk model under sub-independence  }

We consider a portfolio of $n$ sub-independent insurance policies (or risks) $X_i$'s. The aggregate claim amount for the portfolio is
$S=X_1+\cdots+X_n.$
When the number  $n$ of policies  is large, according to the sub-independence version of the central limit theorem, $S$ can approximately follow a normal distribution
$$S\sim N\left(\sum_{i=1}^n E(X_i), \sum_{i=1}^n Var(X_i)\right).$$
This approximation greatly improves computational efficiency, especially for risk assessment of large-scale policy portfolios.
For small-scale combinations, the distribution of $S$ can be accurately solved through convolution operations:
$$F_S(s)=F_{X_1}*F_{X_1}*\cdots * F_{X_n}(s).$$
Or use the moment generating function (MGF):
$M_S(t)=\prod_{i=1}^n M_{X_i}(t);$
Or use the characteristic function (CF)
$\psi_S(t)=\prod_{i=1}^n \psi_{X_i}(t).$


We now present several examples of distributions possessing the reproductive property (i.e., closure under convolution) under sub-independence.

{\bf Example} (i) If $X_1, \cdots, X_n$ are sub-independent and $X_i\sim Bin(n_i,p), i=1,2,\cdots,n$. Then $\sum_{i=1}^n X_i\sim Bin(\sum_{i=1}^n n_i, p).$ \\
(ii) If $X_1, \cdots, X_n$ are sub-independent and $X_i\sim Poi(\lambda_i), i=1,2,\cdots,n$. Then $\sum_{i=1}^n X_i\sim Poi(\sum_{i=1}^n \lambda_i).$\\
(iii) If $X_1, \cdots, X_n$ are sub-independent and $X_i\sim NB(r_i,p), i=1,2,\cdots,n$. Then $\sum_{i=1}^n X_i\sim NB(\sum_{i=1}^n r_i, p).$ \\
(iv) If $X_1, \cdots, X_n$ are sub-independent and $X_i\sim N(\mu_i, \sigma_i^2), i=1,2,\cdots,n$. Then $\sum_{i=1}^n X_i\sim N(\sum_{i=1}^n \mu_i,  \sum_{i=1}^n \sigma_i^2,).$ \\
There are similar conclusions about Gamma, Chi-square and Cauchy  distribution as well.

\subsubsection{The collective risk model under sub-independence  }

The total amount paid (or aggregate losses) incurred by an insurer is expressed as $S_N=X_1+\cdots+X_N,$  where $X_1, X_2,\cdots$ are sub-independent  identically distributed  random variables representing the claim payment amounts, $N$ is the (random) number of claims incurred with the probability mass function $P(N=n)=p_n, n=0,1,\cdots$.  Throughout, we setting $S_0 = 0$, $M_x(t)=E(e^{tX}), P_X(z)=E(z^X)$.
We assume that   $E(e^{itS_n}|N) = E(e^{itS_n})$ for all real $t$. Then,
the distribution of $S$  is given as
$F_S(s)=\sum_{n=0}^{\infty}F_{X}^{*n}(s)p_n,$
the  probability generating function  of $S$  is given as
$P_S(z)=P_N(P_X(z))^n),$
the moment generating function  of $S$  is given as
$M_S(t)=P_N(M_X(t)),$
and the characteristic function  of $S$  is given as
$\psi_S(t)=P_N(\psi_X(t)).$
In particular,
$$ES_N=E(N)E(X),\, Var(S_N)=E(N) Var(X)+(EX)^2Var(N).$$

\begin{remark} In particular, if $N\sim Poission (\lambda)$, i.e. $P(N=n)=\frac{e^{-\lambda}\lambda^n}{n!}, \, n=0,1,2,\cdots$.
The aggregate claim amount $S_N$   follows a compound Poisson distribution with rate $\lambda$ and severity distribution $F_X$:
$S_N\sim$ Compound Poisson $(\lambda, F_X)$. The characteristic function of $S_N$ is given as
$$\psi_{S_N}(t)=\exp(\lambda(\psi_X(t)-1))).$$
From which, we can prove that the sum of $n$ sub-independent compound Poisson distributions is still a compound Poisson distribution.
\end{remark}

\subsection{The sub-independent version of renewal theorems  }\label{intro}

Let $X, X_1,X_2,\cdots$ be a   sequence of nonnegative, independent identically distributed (i.i.d.) random variables (r.v.'s)
with distribution function (d.f.) $F(x) = P(X \le x)$ and   expectation $\mu$.
The sequence $\{S_n, n = 0, 1, 2, . . .\}$ defined by $S_0=0$ and
$S_n = X_1 +\cdots +X_n, n = 1, 2, . . .$,
is called a renewal sequence or a renewal process. The renewal counting process $\{N(t), t \ge 0\}$ associated with the
renewal sequence $\{S_n, n \ge 0\}$ is defined by
$N(t)=\sup\{n\ge 0: S_n\le t \}$.
The most of the important classical
results in renewal theory based on the concept of  independence can be stated in terms of  sub-independent rv's. We only list some of them.

The following result for i.i.d. random variables is well known; see, e.g. Mitov and Omey (2014).

\begin{theorem}   Let $X_1,X_2,\cdots$ be a   sequence of nonnegative, sub-independent identically distributed  random variables with a proper d.f. $F$,  then,

(i) with probability 1,
\begin{eqnarray*}
	\lim_{t\to\infty}\frac{N(t)}{t}=\left\{\begin{array}{ll}
		 \frac{1}{\mu}, \ &{\rm if}\, \, \mu<\infty,\\
		 0,  \ &{\rm if}\,\,  \mu=\infty;
	\end{array}
	\right.
\end{eqnarray*}
(ii) $$P(N(t)=n)=F^{*n}(t)-F^{*n+1}(t), n=1,2,\cdots,$$
where $F^{*n}$ is the $n$-fold convolution of $F$ with itself.
\end{theorem}

The following theorem gives the  results of the process $N(t)$ under the assumption of sub-independence.

\begin{theorem}   Let $X_1,X_2,\cdots$ be a   sequence of nonnegative, sub-independent identically distributed  random variables with a proper d.f. $F$ satisfying $F(0)<1$ and $F(\infty)=1$,  then,
$$U(t):=E(N(t))=\sum_{k=1}^{\infty}F^{*n}(t),$$
 $$U(t)=F(t)+(U*F)(t),$$
 and
 \begin{eqnarray*}
	\lim_{t\to\infty}\frac{U(t)}{t}=\left\{\begin{array}{ll}
		 \frac{1}{\mu}, \ &{\rm if}\, \, \mu<\infty,\\
		 0,  \ &{\rm if}\,\,  \mu=\infty,
	\end{array}
	\right.
\end{eqnarray*}
where $F^{*n}$ is the $n$-fold convolution of $F$ with itself.
If, in addition, $0<\sigma^2=Var(X)<\infty$, then
$$\frac{N(t)-t/\mu}{\sigma\sqrt{\mu^3 t}} \stackrel{\cal{D}}{\longrightarrow} N(0,1), \,as\, t\to\infty.$$
\end{theorem}

\section{Conclusions and prospects }\label{intro}

This paper introduced a new  index  of dependence for testing   sub-independence. The concept of sub-independence,  lies in between independence and uncorrelatedness, is much
weaker than that of independence, is shown to be sufficient to yield the conclusions of  theorems and
results  in these situations which deal with the distribution of the sum of random
variables. Hence, it  providing a more flexible theoretical basis for handling nonlinear and non-normally distributed data.
Sub-independence  can be applied in modern statistical models for complex scenarios where traditional independence assumptions are too strong and only uncorrelated is insufficient to characterize relationships.  It can be expected to demonstrate unique value in the following areas:

{\bf Financial time series modeling}:
In asset return modeling, the independence of sequences often does not hold (such as volatility clustering effects), but assuming complete correlation directly would make the model too complex. Sub-independence can be used to describe weak dependencies between high-order moments (such as volatility), capturing nonlinear dynamic features while maintaining model simplicity. It is commonly used in residual diagnosis and extension of GARCH family models.

{\bf Survival analysis and reliability engineering}:
In multivariate survival data, failure events between individuals may not be completely independent (such as familial genetic influences), but there may not necessarily be strong associations. Sub-independence is used to construct a semi-parametric copula model, allowing for flexible setting of marginal distributions while connecting variables with weaker dependency structures to enhance model robustness.

{\bf Dimensionality reduction and variable selection: 
}
In gene expression data analysis, there are complex regulatory networks among thousands of genes. Sub-independence can be used as a screening criterion to identify gene modules that are not completely independent but have potential functional associations, assisting in the construction of more reasonable biological network models.

{\bf Bayesian networks and causal inference}:
When the data is insufficient to support strong independence testing, secondary independence can be used as an intermediate hypothesis to simplify the learning process of graph model structure. It helps avoid structural misjudgment caused by excessive reliance on chi-square tests or information criteria, especially suitable for small sample or sparse data scenarios.

{\bf Feature engineering in machine learning}:
In ensemble learning and random forests, the selection of feature subsets often relies on the assumption of independence. Introducing sub-independence measures  can more finely control the redundancy between features and improve the model's generalization ability while ensuring diversity.


\noindent{\bf Acknowledgements.}
This research   was supported by the National Natural Science Foundation of China (No. 12071251).\\

\bibliographystyle{model1-num-names}

\end{document}